\documentclass[a4paper,12pt]{article}
\usepackage[utf8]{inputenc}

\usepackage{mathrsfs}
\usepackage{graphicx}
\usepackage{enumerate}
\usepackage{multicol}
\usepackage{color}
\usepackage{amsmath,amssymb,amscd}
\usepackage{amsthm}
\usepackage{mathabx}

\usepackage{pdfpages}
\usepackage{hyperref}

\usepackage{times}
\usepackage[varg]{txfonts}

\usepackage[top=1in, bottom=1in, left=0.75in, right=0.75in]{geometry}

\theoremstyle{plain}
\newtheorem{thm}{Theorem}
\newtheorem{lem}{Lemma}
\newtheorem{cor}{Corollary}
\newtheorem{prop}{Proposition}

\theoremstyle{definition}
\newtheorem{dfn}{Definition}
\newtheorem{ex}{Example}

\theoremstyle{remark}
\newtheorem{rem}{Remark}
\newtheorem*{opp}{Open Problem}
\newtheorem*{ackn}{Acknowledgment}

\newcommand{\C}{\mathbb{C}}

\title{Quadratic functions as solutions of polynomial equations}
\author{Eszter Gselmann and Mehak Iqbal}

\begin{document}
\maketitle

\begin{abstract}
The so-called polynomial equations play an important role both in algebra and in the theory of functional equations. If the unknown functions in the equation are additive, relatively many results are known. However, even in this case, there are a lot of open questions. 
In some specific cases, according to classical results, the unknown additive functions are homomorphisms, derivations, or linear combinations of these. 
The question arises as to whether the solutions can be described even if the unknown functions are not assumed to be additive but to be generalized monomials. As a starting point, in this paper, we will deal with quadratic functions. We aim to show that quadratic functions that are solutions to certain polynomial equations necessarily have a `special' form. Further, we also present a method to determine these special forms. 
\end{abstract}

\section{Introduction and preliminaries}

The description and characterization of different morphisms in given algebraic structures is an important issue in algebra. Further, since these questions mean the fulfillment of certain identities, such investigations can also be relevant from the point of view of the theory of functional equations.

 Recall that if  $R, R'$ are rings, then the mapping $\varphi \colon R\rightarrow R'$ is called a \emph{homomorphism} if
\[
 \varphi(a+b)=\varphi(a)+\varphi(b)
\qquad \left(a, b\in R\right)
\]
and
\[
 \varphi(ab)=\varphi(a)\varphi(b)
\qquad \left(a, b\in R\right).
\]
Furthermore, ,the function $\varphi:R\to R'$ is an
\emph{anti-homomorphism} if
\[
 \varphi(a+b)=\varphi(a)+\varphi(b)
\qquad \left(a, b\in R\right)
\]
and
\[
 \varphi(ab)=\varphi(b)\varphi(a)
\qquad \left(a, b\in R\right).
\]

Let $n\geq 2$ be a fixed positive integer. 
The function $\varphi\colon R\to R'$ is called an $n$-Jordan
homomorphism if
\[
\varphi(a+b)=\varphi(a)+\varphi(b) \qquad \left(a, b\in R\right)
\]
and
\[
 \varphi(a^{n})=\varphi(a)^{n}
\qquad \left(a\in R\right).
\]
In case $n=2$ we speak about homomorphisms
and Jordan homomorphisms, respectively. It was G.~Ancochea who
first dealt with the connection of Jordan homomorphisms and
homomorphisms, see \cite{Anc42}. These results were
generalized and extended in several ways, see for instance
\cite{JR}, \cite{Kap}, \cite{Zel68}. 
According to a classical result \cite{Her} of I.N.~Herstein,  
if $\varphi$ is a Jordan homomorphism of a ring $R$
\emph{onto} a prime ring $R'$ of characteristic different from $2$
and $3$, then either $\varphi$ is a homomorphism or an
anti-homomorphism. 

Besides homomorphisms, derivations also play a key role in the theory of rings and fields. Concerning this notion, we will follow \cite[Chapter 14]{Kuc}. 

Let $Q$ be a ring and let $P$ be a subring of $Q$.
A function $d\colon P\rightarrow Q$ is called a \emph{derivation}\index{derivation} if it is additive,
i.e. 
\[
d(x+y)=d(x)+d(y)
\quad
\left(x, y\in P\right)
\]
and also satisfies the so-called \emph{Leibniz rule}\index{Leibniz rule}, i.e.  equation
\[
d(xy)=d(x)y+xd(y)
\quad
\left(x, y\in P\right). 
\]

It is well-known that in the case of additive functions, Hamel bases play an important role. 
As \cite[Theorem 14.2.1]{Kuc} shows in the case of derivations, algebraic bases are fundamental. 

\begin{thm}\label{T14.2.1}
Let $(\mathbb{K}, +,\cdot)$ be a field of characteristic zero, let $(\mathbb{F}, +,\cdot)$
be a subfield of $(\mathbb{K}, +,\cdot)$, let $S$ be an algebraic base of $\mathbb{K}$ over $\mathbb{F}$,
if it exists, and let $S=\emptyset$ otherwise.
Let $f\colon \mathbb{F}\to \mathbb{K}$ be a derivation.
Then, for every function $u\colon S\to \mathbb{K}$,
there exists a unique derivation $g\colon \mathbb{K}\to \mathbb{K}$
such that $g \vert_{\mathbb{F}}=f$ and $g \vert_{S}=u$.
\end{thm}

Similar to homomorphisms, characterization theorems related to derivations also have extensive literature, see e.g.~the monographs \cite{Kuc, ZarSam75}. 

According to a classical result in connection to derivations, if $\mathbb{F}$ is a subfield of the field $\mathbb{K}$ with characteristic zero, $P\in \mathbb{F}[x]$ is a polynomial and the additive function $a\colon \mathbb{F}\to \mathbb{K}$ fulfills 
\[
   a(P(x))= P'(x)a(x) 
   \qquad 
   \left(x\in \mathbb{F}\right), 
\]
then $a$ is a derivation. 

Note that all the above problems can be viewed as special cases of a more general problem. Indeed, let $\mathbb{K}$ be a field of characteristic zero and $\mathbb{F}\subset \mathbb{K}$ be a subfield, let further $P\in \mathbb{F}[x]$ and $Q\in \mathbb{K}[x_{1}, x_{2}]$ be given polynomials and $a\colon \mathbb{F}\to \mathbb{K}$ be an additive function such that 
\[
a(P(x))=Q(x, a(x)) 
\qquad 
\left(x\in \mathbb{F}\right). 
\]
The question arises: Does the above identity imply that this additive function $a$ has some `special form'? For certain polynomials $P$ and $Q$, in the case of classical results, the unknown additive function $a$ is a homomorphism, a derivation, or a linear combination of these. Naturally, the question arises as to whether this is the case for all polynomials $P$ and $Q$.  
Similar questions can be raised about generalized monomial functions instead of additive functions: assume that $n$ and $k$ are positive integers, $P_{i, j}\in \mathbb{F}[x]$ and $P\in \mathbb{K}[z, x_{1, 1}, \ldots, x_{n, k}]$ are given polynomials for $i=1, \ldots, n; j=1, \ldots, k$. Suppose further that $f_{1}, \ldots, f_{n}\colon \mathbb{F}\to \mathbb{K}$ are generalized monomials (of possibly different degree) such that 
\[
\tag{$\bullet$}\label{bullet}
P(x, f_{1}(P_{1, 1}(x)), \ldots, f_{1}(P_{1, k}(x)), \ldots, f_{n}(P_{n, 1}(x)), \ldots, f_{n}(P_{n, k}(x)))=0
\]
holds for all $x\in \mathbb{F}$? Is it true that the monomial functions here necessarily have a `special' form? If so, is there a method to determine these special forms?

In the case where the unknown generalized monomial functions are additive, some results can be found e.g. in \cite{Eba15, Eba17, Eba18, EbaRieSah, GseKisVin18, GseKisVin19}.  The papers \cite{Gse22, GseIqb23} contain related results, but there the unknown functions were assumed to be quadratic. This paper is a continuation of these.  

We also mention papers \cite{Amo20, BorGar18, BorGar23AM, BorMen23AM, BorMen23AMS} where the authors also dealt with quadratic functions that satisfy additional identities, typically an alternative equation, or a conditional equation. 
As the papers \cite{AicMoo21} and \cite{Almira22} (and also their references) show, such and similar so-called polynomial equations as $(\bullet)$ can appear in many areas of mathematics. Of course, most often in algebra, the theory of functional equations, but also in probability theory or statistics.

One of the main difficulties in solving equation \eqref{bullet} is that it contains a single independent variable but $n$ unknown functions. Thus, the primary goal is usually to be able to increase the `degree of freedom' provided by this independent variable. In other words, we would like to set up an equation for the unknown functions that is equivalent to equation \eqref{bullet} but contains many more independent variables. This is provided by the assumption that the functions $f_{1}, \ldots, f_{n}$ in the equation are \emph{monomial} functions. More precisely, the so-called Polarization formula gives the possibility that (if certain additional conditions are met) the unknown functions in the equation satisfy a Levi-Civita equation 
\[
f(xy)= \sum_{i=1}^{k}g_{i}(x)h_{i}(y)
\]
on $\mathbb{F}^{\times}$, i.e. on the \emph{multiplicative} structure of the domain. Therefore, our main objective in this paper is to determine all those quadratic functions $q$ that satisfy a Levi-Civita equation on the multiplicative structure, i.e., that can be written as
\[
q(xy)=  \sum_{i=1}^{k}g_{i}(x)h_{i}(y)
\]
with some positive integer $k$ and with some appropriate functions $g_{i}, h_{i}$, $i=1, \ldots, k$. For this, those quadratic functions $q$ that satisfy the equations
\[
q(xy)= q(x)q(y) 
\qquad 
\left(x, y\in \mathbb{F}^{\times}\right)
\]
and 
\[
q(xy)= x^{2}q(y)+q(x)y^{2} 
\qquad 
\left(x, y\in \mathbb{F}^{\times}\right), 
\]
respectively, must first be determined.

\subsection*{Generalized polynomial functions}

While proving our results, the so-called Polarization formula for multi-additive functions and the symmetrization method will play a key role. In this subsection, the most important notations and statements are summarized based on the monograph \cite{Sze91}.

\begin{dfn}
 Let $G, S$ be commutative semigroups, $n\in \mathbb{N}$ and let $A\colon G^{n}\to S$ be a function.
 We say that $A$ is \emph{$n$-additive} if it is a homomorphism of $G$ into $S$ in each variable.
 If $n=1$ or $n=2$ then the function $A$ is simply termed to be \emph{additive}
 or \emph{bi-additive}, respectively.
\end{dfn}

The \emph{diagonalization} or \emph{trace} of an $n$-additive
function $A\colon G^{n}\to S$ is defined as
 \[
  A^{\ast}(x)=A\left(x, \ldots, x\right)
  \qquad
  \left(x\in G\right).
 \]
As a direct consequence of the definition each $n$-additive function
$A\colon G^{n}\to S$ satisfies
\[
 A(x_{1}, \ldots, x_{i-1}, kx_{i}, x_{i+1}, \ldots, x_n)
 =
 kA(x_{1}, \ldots, x_{i-1}, x_{i}, x_{i+1}, \ldots, x_{n})
 \qquad 
 \left(x_{1}, \ldots, x_{n}\in G\right)
\]
for all $i=1, \ldots, n$, where $k\in \mathbb{N}$ is arbitrary. The
same identity holds for any $k\in \mathbb{Z}$ provided that $G$ and
$S$ are groups, and for $k\in \mathbb{Q}$, provided that $G$ and $S$
are linear spaces over the rationals. For the diagonalization of $A$
we have
\[
 A^{\ast}(kx)=k^{n}A^{\ast}(x)
 \qquad
 \left(x\in G\right).
\]

The above notion can also be extended for the case $n=0$ by letting 
$G^{0}=G$ and by calling $0$-additive any constant function from $G$ to $S$. 

One of the most important theoretical results concerning
multi-additive functions is the so-called \emph{Polarization
formula}, that briefly expresses that every $n$-additive symmetric
function is \emph{uniquely} determined by its diagonalization under
some conditions on the domain as well as on the range. Suppose that
$G$ is a commutative semigroup and $S$ is a commutative group. The
action of the {\emph{difference operator}} $\Delta$ on a function $f\colon G\to S$ is defined by
\[\Delta_y f(x)=f(x+y)-f(x)
\qquad
\left(x, y\in G\right). \]
Note that the addition in the argument of the function is the
operation of the semigroup $G$ and the subtraction means the inverse
of the operation of the group $S$.

\begin{thm}[Polarization formula]\label{Thm_polarization}
 Suppose that $G$ is a commutative semigroup, $S$ is a commutative group, $n\in \mathbb{N}$.
 If $A\colon G^{n}\to S$ is a symmetric, $n$-additive function, then for all
 $x, y_{1}, \ldots, y_{m}\in G$ we have
 \[
  \Delta_{y_{1}, \ldots, y_{m}}A^{\ast}(x)=
  \left\{
  \begin{array}{rcl}
   0 & \text{ if} & m>n \\
   n!A(y_{1}, \ldots, y_{m}) & \text{ if}& m=n.
  \end{array}
  \right.
 \]

\end{thm}

\begin{cor}
 Suppose that $G$ is a commutative semigroup, $S$ is a commutative group, $n\in \mathbb{N}$.
 If $A\colon G^{n}\to S$ is a symmetric, $n$-additive function, then for all $x, y\in G$
 \[
  \Delta^{n}_{y}A^{\ast}(x)=n!A^{\ast}(y).
\]
\end{cor}

\begin{lem}
\label{mainfact}
  Let $n\in \mathbb{N}$ and suppose that the multiplication by $n!$ is surjective in the commutative semigroup $G$ or injective in the commutative group $S$. Then for any symmetric, $n$-additive function $A\colon G^{n}\to S$, $A^{\ast}\equiv 0$ implies that
  $A$ is identically zero, as well.
\end{lem}

\begin{dfn}
 Let $G$ and $S$ be commutative semigroups, a function $p\colon G\to S$ is called a \emph{generalized polynomial} from $G$ to $S$ if it has a representation as the sum of diagonalizations of symmetric multi-additive functions from $G$ to $S$. In other words, a function $p\colon G\to S$ is a 
 generalized polynomial if and only if, it has a representation 
 \[
  p= \sum_{k=0}^{n}A^{\ast}_{k}, 
 \]
where $n$ is a nonnegative integer and $A_{k}\colon G^{k}\to S$ is a symmetric, $k$-additive function for each 
$k=0, 1, \ldots, n$. In this case, we also say that $p$ is a generalized polynomial \emph{of degree at most $n$}. 

Let $n$ be a nonnegative integer, functions $p_{n}\colon G\to S$ of the form 
\[
 p_{n}= A_{n}^{\ast}, 
\]
where $A_{n}\colon G^{n}\to S$ is symmetric and {n}-additive are the so-called \emph{generalized monomials of degree $n$}. 

\end{dfn}

\begin{rem}
 Generalized monomials 
of degree $0$ are constant functions and generalized monomials of degree $1$ are additive functions. 
Furthermore, generalized monomials of degree $2$ will be termed \emph{quadratic functions}. 
\end{rem}

\subsection*{Polynomial functions}

In the previous subsection, we already introduced a concept for polynomials, namely the notion of generalized polynomials. It is important to emphasize that this is not the only way to define polynomials on groups. 
In Laczkovich \cite{Lac19}, the interested reader can find several concepts. Further, we can read about whether there is a connection between these different concepts and, if so, what kind of connection it is.
As we will see below, the notion of \emph{(normal) polynomials} will also be necessary. The definitions and results recalled here can be found in \cite{Sze91}. 

Throughout this subsection $G$ is assumed to be a commutative group (written additively).

\begin{dfn}
{\it Polynomials} are elements of the algebra generated by additive
functions over $G$. Namely, if $n$ is a positive integer,
$P\colon\mathbb{C}^{n}\to \mathbb{C}$ is a (classical) complex polynomial in
 $n$ variables and $a_{k}\colon G\to \mathbb{C}\; (k=1, \ldots, n)$ are additive functions, then the function
 \[
  x\longmapsto P(a_{1}(x), \ldots, a_{n}(x))
 \]
is a polynomial, also conversely, every polynomial can be
represented in such a form.
\end{dfn}

\begin{rem}
 For easier distinction, in some places, polynomials will be called normal polynomials. 
\end{rem}

\begin{rem}
 We recall that the elements of $\mathbb{N}^{n}$ for any positive integer $n$ are called
 ($n$-dimensional) \emph{multi-indices}.
 Addition, multiplication, and inequalities between multi-indices of the same dimension are defined component-wise.
 Further, we define $x^{\alpha}$ for any $n$-dimensional multi-index $\alpha$ and for any
 $x=(x_{1}, \ldots, x_{n})$ in $\mathbb{C}^{n}$ by
 \[
  x^{\alpha}=\prod_{i=1}^{n}x_{i}^{\alpha_{i}}
 \]
where we always adopt the convention $0^{0}=1$. We also use the
notation $\left|\alpha\right|= \alpha_{1}+\cdots+\alpha_{n}$. With
these notations, any polynomial of degree at most $N$ on the
commutative semigroup $G$ has the form
\[
 p(x)= \sum_{\left|\alpha\right|\leq N}c_{\alpha}a(x)^{\alpha}
 \qquad
 \left(x\in G\right),
\]
where $c_{\alpha}\in \mathbb{C}$ and $a=(a_1, \dots, a_n) \colon
G\to \mathbb{C}^{n}$ is an additive function. Furthermore, the
\emph{homogeneous term of degree $k$} of $p$ is
\[
 \sum_{\left|\alpha\right|=k}c_{\alpha}a(x)^{\alpha} .
\]
\end{rem}

It is easy to see that 
each polynomial, that is, any function of the form 
\[
  x\longmapsto P(a_{1}(x), \ldots, a_{n}(x)), 
 \]
where $n$ is a positive integer,
$P\colon\mathbb{C}^{n}\to \mathbb{C}$ is a (classical) complex
polynomial in
 $n$ variables and $a_{k}\colon G\to \mathbb{C}\; (k=1, \ldots, n)$ are additive functions, is a generalized polynomial. The converse however is in general not true. A complex-valued generalized polynomial $p$ defined on a commutative group $G$ is a polynomial \emph{if and only if} its variety (the linear space spanned by its translates) is of \emph{finite} dimension. 
To clarify the situation, here we also recall Theorem 13.4 from Sz\'ekelyhidi \cite{Sze14}. 
 
 \begin{thm}\label{thm_torsion}
  The torsion-free rank of a commutative group is finite \emph{if and only if} every generalized polynomial on the group is a polynomial. 
 \end{thm}

The notion of derivations can be extended in several ways. We will employ the concept of higher order derivations according to Reich \cite{Rei98} and Unger--Reich \cite{UngRei98}. For further results on characterization theorems on higher order derivations consult e.g. \cite{Eba15, Eba17, EbaRieSah} and 
\cite{GseKisVin18}. 

\begin{dfn}
 Let $\mathbb{F}\subset \mathbb{K}$ be fields with characteristic zero. The identically zero map is the only \emph{derivation of order zero}. For each $n\in \mathbb{N}$, an additive mapping 
 $f\colon \mathbb{F}\to \mathbb{K}$ is termed to be a \emph{derivation of order $n$}, if there exists $B\colon \mathbb{F}\times \mathbb{F}\to \mathbb{K}$ such that 
 $B$ is a bi-derivation of order $n-1$ (that is, $B$ is a derivation of order $n-1$ in each variable) and 
 \[
  f(xy)-xf(y)-f(x)y=B(x, y) 
  \qquad 
  \left(x, y\in \mathbb{F}\right). 
 \]
 The set of $\mathbb{K}$-valued derivations of order $n$ of the field $\mathbb{F}$ will be denoted by $\mathscr{D}_{n}(\mathbb{F}, \mathbb{K})$. 
\end{dfn}

\begin{rem}
\label{pathologic}
Since $\mathscr{D}_{0}(\mathbb{F}, \mathbb{K})=\{0\}$, the only bi-derivation of order zero is the identically zero function, thus $f\in \mathscr{D}_{1}(\mathbb{F}, \mathbb{K})$ if and only if
  \[
   f(xy)=xf(y)+f(x)y 
   \qquad 
   \left(x, y\in \mathbb{F}\right), 
  \]
that is, the notions of first-order derivations and derivations coincide. On the other hand for any $n\in \mathbb{N}$ the set $\mathscr{D}_{n}(\mathbb{F}, \mathbb{K})\setminus \mathscr{D}_{n-1}(\mathbb{F}, \mathbb{K})$ is nonempty because  $d_{1}\circ \cdots \circ d_{n}\in \mathscr{D}_{n}(\mathbb{F}, \mathbb{K})$, but $d_{1}\circ \cdots \circ d_{n}\notin \mathscr{D}_{n-1}(\mathbb{F}, \mathbb{K})$, where $d_{1}, \ldots, d_{n}\in \mathscr{D}_{1}(\mathbb{F}, \mathbb{K})$ are non-identically zero derivations.
\end{rem}

\begin{dfn}
 Let $\mathbb{F}\subset \mathbb{K}$ be fields of characteristic zero. We say that the map 
 $D \colon \mathbb{F}\to \mathbb{K}$ is a \emph{differential operator of order at most $n$} if $D$ is the linear combination, with coefficients from $\mathbb{F}$, of finitely many maps of the form 
 $d_1 \circ \cdots \circ d_k$, where $d_1, \ldots, d_k$ are $\mathbb{K}$-valued derivations on $\mathbb{F}$ and $k\leq n$. If $k = 0$ then we interpret $d_1\circ \cdots \circ d_k$ as the identity function. 
 We denote by $\mathscr{O}_n(\mathbb{F}, \mathbb{K})$ the set of differential operators of order at most $n$ defined on $\mathbb{F}$. We say that the order of a differential operator $D$ is $n$ if $D \in \mathscr{O}_{n}(\mathbb{F}, \mathbb{K})\setminus\mathscr{O}_{n-1}(\mathbb{F}, \mathbb{K})$ (where $\mathscr{O}_{-1}(\mathbb{F}, \mathbb{K})= \emptyset$, by definition).
\end{dfn}


The main result of \cite{KisLac18} is Theorem 1.1 which reads in our settings as follows. 

\begin{thm}\label{theorem_KisLac}
 Let $\mathbb{F}\subset \mathbb{K}$ be fields of characteristic zero and let $n$ be a positive integer. Then, for every function $D \colon \mathbb{F}\to \mathbb{C}$, the
following are equivalent.
\begin{enumerate}[(i)]
\item $D\in \mathscr{D}_{n}(\mathbb{F}, \mathbb{K})$
\item $D\in \mathrm{cl}\left(\mathscr{O}_{n}(\mathbb{F}, \mathbb{K})\right)$
\item $D$ is additive on $\mathbb{F}$, $D(1) = 0$, and $D/\mathrm{id}$, as a map from the group $\mathbb{F}^{\times}$ to $\mathbb{K}$, is a generalized polynomial of degree at most $n$. Here $\mathrm{id}$ stands for the identity map defined on $\mathbb{F}$. 
\end{enumerate}
\end{thm}

\subsection*{Moment sequences}

A \emph{composition} of a nonnegative integer $n$ is a sequence of nonnegative integers $\alpha= \left(\alpha_{k}\right)_{k\in \mathbb{N}}$ such that 
\[
 n= \sum_{k=1}^{\infty}\alpha_{k}. 
\]
For a positive integer $r$, an \emph{$r$-composition} of a nonnegative integer $n$ is a composition 
$\alpha= \left(\alpha_{k}\right)_{k\in \mathbb{N}}$ with  $\alpha_{k}=0$ for $k>r$. 

Given a sequence of variables $x=(x_{k})_{k\in \mathbb{N}}$ and compositions $\alpha= \left(\alpha_{k}\right)_{k\in \mathbb{N}}$ and 
$\beta= \left(\beta_{k}\right)_{k\in \mathbb{N}}$ we define 
\[
 \alpha!=\prod_{k=1}^{\infty}\alpha_{k},\quad \left| \alpha\right| = \sum_{k=1}^{\infty}\alpha_{k}, \quad 
 x^{\alpha}=\prod_{k=1}^{\infty}x_{k}^{\alpha_{k}},\quad \binom{\alpha}{\beta}= \prod_{k=1}^{\infty}\binom{\alpha_{k}}{\beta_{k}}.
\]
Furthermore, 
$\beta \leq \alpha$ means that $\beta_{k}\leq \alpha_{k}$ for all $k\in \mathbb{N}$ and 
$\beta < \alpha$ stands for $\beta \leq \alpha$  and $\beta \neq \alpha$.

\begin{dfn}\label{dfn_moment}
Let $G$ be an Abelian group, $r$ a positive integer, and for each multi-index $\alpha$ in $\mathbb{N}^r$ 
let $f_{\alpha}\colon G\to \mathbb{C}$ be continuous function. We say that $(f_{\alpha})_{\alpha \in \mathbb{N}^{r}}$ is a \emph{generalized  moment sequence of rank $r$}, if 
\begin{equation}\label{Eq3}
f_{\alpha}(x+y)=\sum_{\beta\leq \alpha} \binom{\alpha}{\beta} f_{\beta}(x)f_{\alpha-\beta}(y)
\end{equation}
holds whenever $x,y$ are in $G$. The function $f_0$, where $0$ is the zero element in $\mathbb{N}^r$, is called the {\it generating function} of the sequence.
\end{dfn}

\begin{rem}
\begin{enumerate}[(a)]
\item For $r=1$, instead of multi-indices, we have `ordinary' indices and \eqref{Eq3} is nothing but
 \[
  f_{\alpha}(x+y)= \sum_{\beta=0}^{\alpha}\binom{\alpha}{\beta}f_{\beta}(x)f_{\alpha-\beta}(y) 
  \qquad 
 \]
for each $x, y$ in $G$ and nonnegative integer $\alpha$, yielding that generalized moment functions of rank $1$ are moment sequences. 
\item For $\alpha=(0, \ldots,0)$ we have
$$
f_{0, \ldots,0}(x+y)=f_{0, \ldots,0}(x)\cdot f_{0, \ldots,0}(y) 
\qquad 
$$
for each $x, y$ in $G$ hence $f_{0, \ldots,0}=m$ is an exponential, or identically zero. 
An easy computation shows that if $f_{0, \ldots,0}$ is identically zero, then for any multi-index $\alpha\in \mathbb{N}^{r}$ we have $f_{\alpha}\equiv 0$. Thus we always assume that the generating function is not identically zero, hence it is always an exponential. 
\item The sequence of functions $(f_{\alpha})_{\alpha \in \mathbb{N}^{r}}$ is a generalized moment sequence of rank $r$ associated with the nonzero exponential $f_{0, \ldots, 0}=m$ \emph{if and only if} 
$({f_{\alpha}}/{m})_{\alpha \in \mathbb{N}^{r}}$ is a generalized moment sequence of rank $r$ associated with the exponential which is identically one. 
\item Generalized moment sequences of rank $r$ are  exponential monomials. 
\end{enumerate}
\end{rem}

As a supplement to part (d), we recall Theorem 2 from \cite{FecGseSze21}.  In the statement below, $B_{\alpha}$ stands for the Bell polynomial corresponding to the multi-index $\alpha\in \mathbb{N}^{r}$. 

\begin{thm}\label{thm_moment}
Let $G$ be a commutative group, $r$ a positive integer, and for each $\alpha$ in $\mathbb{N}^{r}$, 
 let $f_{\alpha}\colon G\to \mathbb{K}$ be a function. If the sequence of functions 
 $(f_{\alpha})_{\alpha\in \mathbb{N}^{r}}$ forms a generalized moment sequence of rank $r$, then there exists an   exponential $m\colon G\to \mathbb{K}$ and a sequence of $\mathbb{K}$-valued additive functions $a= (a_{\alpha})_{\alpha\in \mathbb{N}^{r}}$ 
 such that  for every multi-index $\alpha$ in $\mathbb{N}^{r}$ and $x$ in $G$ we have   
 \[
  f_{\alpha}(x)=B_{\alpha}(a(x))m(x). 
 \]
And also conversely, if $r$ is a positive integer, $m\colon G\to \C$ is an exponential and 
 $a= \left(a_{\alpha}\right)_{\alpha\in \mathbb{N}^{r}}$ is a sequence of complex-valued additive functions and for all $\alpha\in \mathbb{N}^{r}$, we define the function  by the above formula, then $(f_{\alpha})_{\alpha\in \mathbb{N}^{r}}$ forms a generalized moment sequence of rank $r$ associated with the exponential $m$. 
\end{thm}

\section{Results}

In this section let $\mathbb{F}$ and $\mathbb{K}$ be fields with zero characteristic such that $\mathbb{F}\subset \mathbb{K}$. 
The following lemma is immediate. 

\begin{lem}\label{lem_variety}
Let us endow $\mathbb{F}^{\times}$ and also $\mathbb{K}$ with the discrete topology and let $\mathscr{C}(\mathbb{F}^{\times}, \mathbb{K})$ denote the linear space of all those functions $f\colon \mathbb{F}^{\times}\to \mathbb{K}$ that are continuous. Then the set 
\[
\mathscr{V}= \left\{f\vert_{\mathbb{F}^{\times}}\, \vert \, f\colon \mathbb{F}\to \mathbb{K} \; \text{is quadratic} \right\}
\]
is a closed, translation-invariant linear subspace of $\mathscr{C}(\mathbb{F}^{\times}, \mathbb{K})$. So $\mathscr{V}\subset \mathscr{C}(\mathbb{F}^{\times}, \mathbb{K})$ is a \emph{variety}. Note the the translate of a function $f\in \mathscr{C}(\mathbb{F}^{\times}, \mathbb{K})$ by an element $y\in \mathbb{F}^{\times}$ is defined as 
\[
(\tau_{y} f) = f(x\cdot y) 
\qquad 
\left(x\in \mathbb{F}^{\times}\right). 
\]
\end{lem}

\begin{ex}\label{Ex1}
Let $\varphi_{1}, \varphi_{2}\colon \mathbb{F}\to \mathbb{K}$ be homomorphisms and let us consider the function $q\colon \mathbb{F}\to \mathbb{K}$ defined by 
\[
q(x)= \varphi_{1}(x)\varphi_{2}(x) 
\qquad 
\left(x\in \mathbb{F}\right). 
\]
Then $q$ is the trace of the symmetric and bi-additive mapping 
\[
B(x, y)= \varphi_{1}(x)\varphi_{2}(y)+\varphi_{1}(y)\varphi_{2}(x) 
\qquad 
\left(x, y\in \mathbb{F}\right). 
\]
So $q$ is a quadratic function. 
Further, we have 
\begin{multline*}
q(xy)= \varphi_{1}(xy)\varphi_{2}(xy)
=
(\varphi_{1}(x)\varphi_{1}(y))\cdot (\varphi_{2}(x)\varphi_{2}(y))
\\
=
(\varphi_{1}(x)\varphi_{2}(x)) \cdot (\varphi_{1}(y)\varphi_{2}(y))
=
q(x)\cdot q(y) 
\qquad 
\left(x, y\in \mathbb{F}\right). 
\end{multline*}
\end{ex}

\begin{ex}\label{Ex2}
Let $d\colon \mathbb{F}\to \mathbb{K}$ be a derivation and define the quadratic mapping $q\colon \mathbb{F}\to \mathbb{K}$ by 
\[
q(x)= d(x^{2}) 
\qquad 
\left(x\in \mathbb{F}\right). 
\]
An easy computation shows that in this case, $q$ determines uniquely a symmetric and bi-additive mapping, namely 
\[
B(x, y)= d(xy) 
\qquad 
\left(x, y\in \mathbb{F}\right). 
\]
Moreover, we have 
\begin{align*}
    \frac{q(xy)}{x^{2}y^{2}}&= \frac{1}{x^{2}y^{2}}d((xy)^{2})
=\frac{1}{x^{2}y^{2}} \cdot 2xy d(xy)
=\frac{2}{xy}\left[xd(y)+d(x)y\right]
\\
&= 2 \frac{d(x)}{x}+2\frac{d(y)}{y}
= \frac{d(x^{2})}{x^{2}}+\frac{d(y^{2})}{y^{2}}  \\
& = \frac{q(x)}{x^{2}}+\frac{q(y)}{y^{2}}
\end{align*}
for all $x, y\in \mathbb{F}^{\times}$. 
\end{ex}

\begin{ex}
Analogously, if $n\geq 2$ is a positive integer then the mapping $f\colon \mathbb{F}\to \mathbb{K}$ defined by 
\[
f(x)= d(x^{n}) 
\qquad 
\left(x\in \mathbb{F}\right)
\]
is a (generalized) monomial of degree $n$. The symmetric, $n$-additive mapping $B_{n}\colon \mathbb{F}^{n}\to \mathbb{K}$ determined by $f$ is 
\[
B_{n}(x_{1}, \ldots, x_{n})= d(x_{1} \cdots x_{n}) 
\qquad 
\left(x_{1}, \ldots, x_{n}\in \mathbb{F}\right). 
\]

Further 
\begin{align*}
\frac{f(xy)}{(xy)^{n}}&= \frac{1}{(xy)^{n}}d((xy)^{n})= \frac{1}{(xy)^{n}} \cdot n(xy)^{n-1}d(xy) \\
&= \frac{n}{xy} \left[d(x)y+xd(y)\right] = n \dfrac{d(x)}{x}+n\frac{d(y)}{y} 
= \frac{d(x^{n})}{x^{n}}+ \frac{d(y^{n})}{y^{n}} \\
&= \frac{f(x)}{x^{n}}+ \frac{f(y)}{y^{n}}
\end{align*}
for all $x, y\in \mathbb{F}^{\times}$. 
\end{ex}

For any positive integer $n\geq 2$, define the mapping $\pi_{n}\colon \mathbb{F}\to \mathbb{K}$ by 
\[
\pi_{n}(x)=x^{n} 
\qquad 
\left(x\in \mathbb{F}\right). 
\]
Observe that in this case, the above examples show that if we consider the (generalized) monomial of degree $n$ 
\[
f(x)= d(x^{n}) 
\qquad 
\left(x\in \mathbb{F}\right), 
\]
then the mapping $\dfrac{f}{\pi_{n}}\colon \mathbb{F}^{\times} \to  \mathbb{K}$ will be additive on the Abelian group $(\mathbb{F}^{\times}, \cdot)$, that is, we have 
\[
\dfrac{f}{\pi_{n}}(xy)= \dfrac{f}{\pi_{n}}(x)+\dfrac{f}{\pi_{n}}(y)
\]
for all $x, y\in \mathbb{F}^{\times}$.

With the notations and assumptions of Lemma \ref{lem_variety}, Example \ref{Ex1} shows that the mapping $q$ defined with the aid of the homomorphisms $\varphi_{1}$ and $\varphi_{2}$ by the formula 
\[
q(x)= \varphi_{1}(x)\varphi_{2}(x) 
\qquad 
\left(x\in \mathbb{F}\right), 
\]
is an exponential in the variety $\mathscr{V}$. 

Similarly, Example \ref{Ex2} shows that the mapping $q$ defined with the help of the derivation $d\colon \mathbb{F}\to \mathbb{K}$ by 
\[
q(x)= d(x^{2}) 
\qquad 
\left(x\in \mathbb{F}\right)
\]
is a moment function of degree $1$ corresponding to the exponential $q_{0}(x)=x^{2}\; (x\in \mathbb{F})$. 

In what follows, we study the converse direction.

\subsection*{Quadratic functions that are multiplicative on the multiplicative group $\mathbb{F}^{\times}$}

\begin{thm}\label{thm_mult}
Let $\mathbb{K}$ be a field of characteristic zero and $\mathbb{F}\subset \mathbb{K}$ be a subfield. Let further $q\colon \mathbb{F}\to \mathbb{K}$ be a quadratic function. 
If the mapping $q$ is multiplicative on $\mathbb{F}^{\times}$, that is, 
\[
q(xy)= q(x)q(y)
\]
holds for all $x, y\in \mathbb{F}^{\times}$, then there exist homomorphisms $\varphi_{1}, \varphi_{2}\colon \mathbb{F}\to \mathbb{K}$ such that 
\[
q(x)= c\cdot \varphi_{1}(x)\varphi_{2}(x)
\qquad 
\left(x\in \mathbb{F}\right), 
\]
where $c\in \left\{ 0, 1\right\}$. 
\end{thm}

\begin{proof}
Let $\mathbb{K}$ be a field of characteristic zero and $\mathbb{F}\subset \mathbb{K}$ be a subfield. 
Suppose further that the quadratic mapping $q\colon \mathbb{F}\to \mathbb{K}$ is multiplicative on $\mathbb{F}^{\times}$, that is, we have 
\[
q(xy)= q(x)q(y)
\]
holds for all $x, y\in \mathbb{F}^{\times}$. Let $B\colon \mathbb{F}\times \mathbb{F}\to \mathbb{C}$ be the uniquely determined symmetric bi-additive mapping for which we have 
\[
B(x, x)= q(x) 
\qquad 
\left(x\in \mathbb{F}\right). 
\]
The assumption that the quadratic mapping $q$ is multiplicative on $\mathbb{F}^{\times}$ yields for the mapping $B$ that 
\[
B(xy, xy)= B(x, x)B(y, y) 
\qquad 
\left(x, y\in \mathbb{F}^{\times}\right), 
\]
especially, we have 
\[
B(x^{2}, x^{2})= B(x, x)^{2} 
\qquad 
\left(x\in \mathbb{F}^{\times}\right). 
\]
Define the mapping $B_{4}$ on $\mathbb{F}^{4}$ by 
\begin{multline*}
    B_{4}(x_1, x_2, x_3, x_4)=
   B\left(x_{1}\,x_{4} , x_{2}\,x_{3}\right)+B\left(x_{1}\,x_{3} , 
 x_{2}\,x_{4}\right)+B\left(x_{1}\,x_{2} , x_{3}\,x_{4}\right)
 \\-B\left(x_{1} , x_{2}\right)\,B\left(x_{3} , 
 x_{4}\right)-B\left(x_{1}
  , x_{3}\right)\,B\left(x_{2} , x_{4}\right)-B\left(x_{1} , x_{4}
 \right)\,B\left(x_{2} , x_{3}\right)\\
 \qquad 
 \left(x_{1}, x_{2}, x_{3}, x_{4}\in \mathbb{F}\right). 
\end{multline*}
Since $B$ is a symmetric and bi-additive mapping, $B_{4}$ is a symmetric and $4$-additive mapping. Further the trace of $B_{4}$ is 
\[
B_{4}(x, x, x, x)= 3 \left(B(x^{2}, x^{2})- B(x, x)^{2} \right)=0
\]
for all $x\in \mathbb{F}^{\times}$. Thus $B_{4}$ vanishes identically on $\mathbb{F}^{4}$. 
Especially, 
\[
0=B_{4}(1, 1, 1, 1)=
-3\,\left(B\left(1 , 1\right)-1\right)\,B\left(1 , 1\right). 
\]
Thus $q(1)=B(1, 1)\in \left\{ 0, 1\right\}$. 
Further, for all $x\in \mathbb{F}^{\times}$ we have 
\[
0= B_{4}(x, 1, 1, 1)=
-3\,\left(B\left(1 , 1\right)-1\right)\,B\left(x , 1\right). 
\]
Finally, 
\[
0=B_{4}(x, x, 1, 1)=
B\left(x^2 , 1\right)+(2-B\left(1 , 1\right))\,B\left(x , x\right)-2\,B\left(x , 1\right)^2
\]
for all $x\in \mathbb{F}^{\times}$. 
Firstly assume that $q(1)=B(1, 1)=0$. Then 
\[
B(x, 1)=0 
\qquad 
\left(x\in \mathbb{F}\right). 
\]
Therefore 
\[
2B(x, x)= 2B(x, 1)^{2}-B(x^{2}, 1)=0
\]
for all $x\in \mathbb{F}^{\times}$. This means that $q$ is identically zero. 

Assume now that $q(1)=B(1, 1)=1$. Consider the additive function $a\colon \mathbb{F}\to \mathbb{K}$ defined by 
\[
a(x)= B(x, 1) 
\qquad 
\left(x\in \mathbb{F}\right). 
\]
Then identity $B_{4}(x, x, 1, 1)=0$ can be re-formulated as 
\[
B(x, x)= 2a(x)^{2}-a(x^{2}) 
\qquad 
\left(x\in \mathbb{F}^{\times}\right), 
\]
from which 
\[
B(x, y)= 2a(x)a(y)-a(xy)
\qquad 
\left(x, y\in \mathbb{F}\right)
\]
and 
\[
q(x)= 2a(x)^{2}-a(x^2)
\qquad 
\left(x\in \mathbb{F}\right)
\]
follow. 
Observe that by the definition of the function $a$, 
\[
a(x)= B(x, 1)= 2a(x)a(1)-a(x)= (2a(1)-1)a(x) 
\qquad 
\left(x\in \mathbb{F}\right). 
\]
Thus $a(1)= 1$, otherwise the function $a$ and thus the $q$ function are identically zero. 
Since we have 
\[
B(x^{2}, x^{2})= B(x, x)^{2}, 
\]
the additive function $a$ has to fulfill 
\[
-a\left(x^4\right)+a\left(x^2\right)^2+4\,a\left(x\right)^2\,a
 \left(x^2\right)-4\,a\left(x\right)^4=0
\]
for all $x\in \mathbb{F}$. 
Again, the left-hand side of this equation is the trace of the symmetric and $4$-additive mapping $A_{4}$ defined on $\mathbb{F}^{4}$ by 
\begin{multline*}
A_{4}(x_1, x_2, x_3, x_4)
=
-a\left(x_{1}\,x_{2}\,x_{3}\,x_{4}\right)+\frac{1}{3}\left[ a\left(x_{1}\,x_{2}
 \right)\,a\left(x_{3}\,x_{4}\right)+a\left(x_{1}\,x_{3}\right)\,a
 \left(x_{2}\,x_{4}\right)+a\left(x_{2}\,x_{3}\right)\,a\left(x_{1}\,
 x_{4}\right)\right]
 \\
 +\frac{2}{3}\,\left[a\left(x_{1}\right)\,a\left(x_{2}
 \right)\,a\left(x_{3}\,x_{4}\right)+a\left(x_{1}\right)\,a\left(
 x_{3}\right)\,a\left(x_{2}\,x_{4}\right)+a\left(x_{2}\right)\,a
 \left(x_{3}\right)\,a\left(x_{1}\,x_{4}\right)
 \right. 
 \\
 \left. +a\left(x_{1}\right)\,
 a\left(x_{2}\,x_{3}\right)\,a\left(x_{4}\right)+a\left(x_{2}\right)
 \,a\left(x_{1}\,x_{3}\right)\,a\left(x_{4}\right)+a\left(x_{1}\,
 x_{2}\right)\,a\left(x_{3}\right)\,a\left(x_{4}\right)\right]
 \\
 -4\,a\left(x_{1}\right)\,a\left(x_{2}\right)\,a\left(x_{3}\right)
 \,a\left(x_{4}\right)
\end{multline*}
Therefore we have
\[
    -\,a\left(x\,y\,z\right)+\,a\left(x\right)\,a\left(y\,z\right)+
 \,a\left(y\right)\,a\left(x\,z\right)+\left(\,a\left(x\,y\right)
 -2  \,a\left(x\right)\,a\left(y\right)\right)\,a\left(z\right)=0
\]
for all $x, y, z\in \mathbb{F}$. For any fixed $z\in \mathbb{F}$ consider the additive function $A_{z}\colon \mathbb{F}\to \mathbb{K}$ defined by
\[
A_{z}(x)=a(xz)-a(x)a(z) 
\qquad 
\left(x\in \mathbb{F}\right). 
\]
The above equation, in terms of this function $A_{z}$ is 
\[
A_{z}(xy)=a(x)A_{z}(y)+ A_{z}(x)a(y) 
\qquad
\left(x, y\in \mathbb{F}\right). 
\]
Thus $A_{z}$ is a normal exponential polynomial of degree at most two on the multiplicative group $\mathbb{F}^{\times}$. Further, 
\begin{enumerate}[(i)]
   \item either for all $z\in \mathbb{F}$ the system $\left\{ a, A_{z}\right\}$ is linearly dependent 
    \item or there exists an element $z\in \mathbb{F}$ such that the system $\left\{ a, A_{z}\right\}$ is linearly independent. 
\end{enumerate}
If we have (i), then 
\[
A_{z}(x)= c(z)a(x)
\]
holds for all $x, z\in \mathbb{F}^{\times}$. Due to the definition of the function $A_{z}$ this means however that 
\[
A(xz)= (c(z)+a(z))a(x) 
\qquad 
\left(x, z\in \mathbb{F}\right). 
\]
Thus $a$ is a constant multiple of an exponential of the Abelian group $\mathbb{F}^{\times}$. 
Thus because of the additivity of $a$, yields that $a$ is a constant multiple of a homomorphism $\varphi\colon \mathbb{F}\to \mathbb{K}$. However, $a(1)=1$, therefore $a=\varphi$. 

If we have (ii), then exists an element $z\in \mathbb{F}$ such that the system $\left\{ a, A_{z}\right\}$ is linearly independent. In this case, however, not only $A_{z}$, but also $a$ is a normal exponential polynomial of degree at most two on the multiplicative group $\mathbb{F}^{\times}$. 
Thus 
\begin{enumerate}[(a)]
    \item either 
    \[
    a(x)= (\alpha_{1} l(x)+\alpha_{2}) m(x) 
    \qquad 
    \left(x\in \mathbb{F}\right)
    \]
    \item or 
    \[
    a(x)= \alpha_{1}m_{1}(x)+ \alpha_{2}m_{2}(x)
    \qquad 
    \left(x\in \mathbb{F}\right)
    \]
\end{enumerate}
with some constants $\alpha_{1}, \alpha_{2}\in \mathbb{K}$ and with an additive function $l\colon \mathbb{F}^{\times}\to \mathbb{K}$ and exponentials $m, m_{1}, m_{2}\colon \mathbb{F}^{\times}\to \mathbb{K}$. 
We emphasize that now we work on the multiplicative group $\mathbb{F}^{\times}$. Therefore, e.g. the fact that $l$ is an additive function means that 
\[
l(xy)= l(x)+l(y) 
\qquad 
\left(x, y\in \mathbb{F}^{\times}\right)
\]
and similarly, the fact that $m$ is exponential means that
\[
m(xy)= m(x)m(y)
\qquad 
\left(x, y\in \mathbb{F}^{\times}\right). 
\]
Using these representations in the equation for the function $q$, it turns out that in case (a) the function $l$ is identically zero, and $\alpha_{2}=1$, in case (b) $\alpha_{1}, \alpha_{2}= \frac{1}{2}$ and $m_{1}, m_{2}$ are homomorphisms, due to the additivity of $a$. 
Thus there exist homomorphisms $\varphi_{1}, \varphi_{2}\colon \mathbb{F}\to \mathbb{K}$ such that 
\[
a(x)= \dfrac{\varphi_{1}(x)+\varphi_{2}(x)}{2} 
\qquad 
\left(x\in \mathbb{F}\right). 
\]
Note, however, that in this case we have 
\begin{align*}
q(x)= 2a(x)^{2}-a(x^2) &=
2\left(\dfrac{\varphi_{1}(x)+\varphi_{2}(x)}{2} \right)^{2}-\dfrac{\varphi_{1}(x^2)+\varphi_{2}(x^2)}{2} \\
&= \frac{\varphi_{1}(x)^2+2\varphi_{1}(x)\varphi_{2}(x)+\varphi_{2}(x)^{2}}{2}-\dfrac{\varphi_{1}(x)^2+\varphi_{2}(x)^2}{2}\\
& = \varphi_{1}(x)\varphi_{2}(x)
\end{align*}
for all $x\in \mathbb{F}$. 
\end{proof}

\subsection*{Quadratic functions that are additive on the multiplicative group $\mathbb{F}^{\times}$}

\begin{thm}\label{thm_add}
Let $\mathbb{K}$ be a field of characteristic zero and $\mathbb{F}\subset \mathbb{K}$ be a subfield. Let further $q\colon \mathbb{F}\to \mathbb{K}$ be a quadratic function. 
If the mapping $\dfrac{q}{\pi_{2}}$ is additive on $\mathbb{F}^{\times}$, that is, 
\[
q(xy)= q(x)y^{2}+x^{2}q(y)
\]
holds for all $x, y\in \mathbb{F}^{\times}$, then there exists a second order derivation $d\colon \mathbb{F}\to \mathbb{K}$ such that 
\[
q(x)= 4xd(x)-d(x^{2})
\qquad 
\left(x\in \mathbb{F}\right). 
\]
\end{thm}

\begin{proof}
Let $\mathbb{F}\subset \mathbb{K}$ be fields with characteristic zero and assume that $q\colon \mathbb{F}\to \mathbb{K}$ is a quadratic mapping which is additive on $\mathbb{F}^{\times}$. In other words, suppose that the quadratic mapping $q$ also fulfills 
\[
q(xy)= q(x)y^{2}+x^{2}q(y)
\]
for all $x, y\in \mathbb{F}^{\times}$. 

Since the mapping $q$ is quadratic, there exists a uniquely determined symmetric and bi-additive mapping $B\colon \mathbb{F}\times \mathbb{F}\to \mathbb{K}$ such that 
\[
B(x, x)= q(x) 
\qquad 
\left(x\in \mathbb{F}\right). 
\]
Then for this symmetric and bi-additive mapping $B$, we have 
\[
B(xy, xy)= y^{2}B(x, x)+x^{2}B(y, y) 
\qquad 
\left(x, y\in \mathbb{F}\right). 
\]
Let us consider the mapping $B_{4}\colon \mathbb{F}^{4}\to \mathbb{K}$ defined by 
\begin{multline*}
  B_{4}(x_1, x_2, x_3, x_4)
=
\frac{1}{3}\left\{{B\left(x_{1}\,x_{4} , x_{2}\,x_{3}\right)+B\left(x_{1}\,x_{3} , 
 x_{2}\,x_{4}\right)+B\left(x_{1}\,x_{2} , x_{3}\,x_{4}\right)}\right\}
 \\
 -\frac{1}{3}\left\{x_{1}\,B\left(x_{3} , x_{2}\right)\,x_{4}+x_{2}\,B\left(x_{3}
  , x_{1}\right)\,x_{4}+B\left(x_{1} , x_{2}\right)\,x_{3}\,x_{4} 
  \right. 
  \\
  \left. 
  +
 x_{1}\,x_{2}\,B\left(x_{3} , x_{4}\right)+x_{1}\,B\left(x_{2} , 
 x_{4}\right)\,x_{3}+B\left(x_{1} , x_{4}\right)\,x_{2}\,x_{3}\right\}
 \\
 \left(x_1, x_2, x_3, x_4\in \mathbb{F}\right). 
  \end{multline*}
Since $B$ is symmetric and bi-additive, $B_{4}$ is symmetric and $4$-additive. Further, its trace is 
\[
B_{4}(x, x, x, x)= B\left(x^2 , x^2\right)-2\,x^2\,B\left(x , x\right)=0
\]
for all $x \in \mathbb{F}$. Thus $B_{4}$ is identically zero on $\mathbb{F}^{4}$. 
Especially, we have 
\[
B_{4}(1, 1, 1, 1)= -B(1, 1)=0, 
\]
so $B(1, 1)= q(1)=0$. 
Further, we have for all $x, y\in \mathbb{F}$ that 
\begin{multline*}
0= B_{4}(x, y, 1, 1)
\\
= 
B\left(x\,y , 1\right)-x\,B\left(y , 1\right)-B\left(x , 1\right)\,
 y-B\left(1 , 1\right)\,x\,y-B\left(1 , x\right)\,y+B\left(x , y
 \right)-B\left(1 , y\right)\,x
 \\
 =
B(xy, 1)-2xB(y, 1) -2yB(x, 1)+B(x, y). 
\end{multline*}
Define the additive mapping $a\colon \mathbb{F}\to \mathbb{K}$ by 
\[
a(x)= B(x, 1) 
\qquad 
\left(x\in \mathbb{F}\right)
\]
to deduce that the latter identity is 
\[
B(x, y)= 2xa(y)+2ya(y)-a(xy) 
\qquad 
\left(x, y\in \mathbb{F}\right). 
\]
Observe that this already shows that 
\[
q(x)= B(x, x)= 4xa(x)-a(x^{2}) 
\qquad 
\left(x\in \mathbb{F}\right). 
\]
Now we show that $a$ is a second-order derivation. 
Since the quadratic function $q$ fulfills 
\[
q(x^{2})= 2x^{2}q(x) 
\qquad 
\left(x\in \mathbb{F}\right), 
\]
for the additive function $a$ we have 
\[
a\left(x^4\right)-6\,x^2\,a\left(x^2\right)+8\,x^3\,a\left(x\right)=0
\]
for all $x\in \mathbb{F}$. 
Using \cite[Corollary 4]{GseKisVin18}, the additive function $a$ has to be a derivation of order two. 
Therefore, there exists a mapping $d\in \mathscr{D}_{2}(\mathbb{F}, \mathbb{K})$ such that 
\[
q(x)= B(x, x)= 4xd(x)-d(x^{2})
\]
holds for all $x\in \mathbb{F}$. 
\end{proof}

\begin{rem}
Note that the quadratic mapping that appears in Example \ref{Ex2} is covered in Theorem \ref{thm_add}.  Indeed, if $d\colon \mathbb{F}\to \mathbb{K}$ is a derivation, then $d\in \mathscr{D}_{1}(\mathbb{F}, \mathbb{K}) \subset \mathscr{D}_{2}(\mathbb{F}, \mathbb{K})$. So $d$ is a derivation of order two, too. 
Thus
\[
q(x)= 4x d(x)-d(x^{2})= 2 d(x^2)-d(x^{2})=d(x^{2}) 
\qquad 
\left(x\in \mathbb{F}\right), 
\]
showing that mappings appearing in Example \ref{Ex2} can indeed be written as mappings appearing in Theorem \ref{thm_add}. 

However, Theorem \ref{thm_add} shows that the equation 
\[
q(xy)= x^{2}q(y)+y^{2}q(x)=0 
\qquad 
\left(x, y\in \mathbb{F}\right)
\]
has other quadratic solutions than those found in Example \ref{Ex2}. 
Indeed, if $d_{1}, d_{2}\colon \mathbb{F}\to \mathbb{K}$ are non-identically zero derivations, then $d_{1}\circ d_{2}\in \mathscr{D}_{2}(\mathbb{F}, \mathbb{K})\setminus \mathscr{D}_{1}(\mathbb{F}, \mathbb{K})$. Then the mapping $q\colon \mathbb{F}\to \mathbb{K}$ defined by 
\[
q(x)= 4x d_{1}\circ d_{2}(x)-d_{1}\circ d_{2}(x^{2}) 
\qquad 
\left(x\in \mathbb{F}\right)
\]
is quadratic. 
Further, we have 
\begin{align*}
    q(x^{2})&= 4x^{2} d_{1}\circ d_{2}(x^{2})-d_{1}\circ d_{2}(x^{4}) \\
    &= 4x^{2}\left(2xd_{1}\circ d_{2}(x)+2d_{1}(x)d_{2}(x)\right)-\left(4x^{3}d_{1}\circ d_{2}(x)+12x^{2}d_{1}(x)d_{2}(x)\right)\\
    & = 4x^{3}d_{1}\circ d_{2}(x)-4x^{2}d_{1}(x)d_{2}(x)
\end{align*}
and 
\begin{align*}
    2x^{2}q(x)&= 2x^{2}\left(4x d_{1}\circ d_{2}(x)-d_{1}\circ d_{2}(x^{2}) \right)\\
    &= 8x^{3}d_{1}\circ d_{2}(x)-2x^{2}\left(2xd_{1}\circ d_{2}(x)+2d_{1}(x)d_{2}(x)\right)\\
    & = 4x^{3}d_{1}\circ d_{2}(x)-4x^{2}d_{1}(x)d_{2}(x)
\end{align*}
for all $x\in \mathbb{F}$. 
Thus 
\[
q(x^{2})=2x^{2}q(x) 
\qquad 
\left(x\in \mathbb{F}\right). 
\]
Note, however, that due to the methods described in the proof of Theorem \ref{thm_add}, this equation is equivalent to 
\[
q(xy)= x^{2}q(y)+y^{2}q(x) 
\qquad 
\left(x, y\in \mathbb{F}\right). 
\]
\end{rem}

In view of Theorem \ref{thm_mult}, if $\varphi\colon \mathbb{F}\to \mathbb{K}$ is a homomorphism then the mapping $q$ defined on $\mathbb{F}$ by 
\[
q(x)= \varphi(x)^{2} 
\qquad 
\left(x\in \mathbb{F}\right)
\]
is quadratic and is also an exponential on the multiplicative group $\mathbb{F}^{\times}$ 
Thus, using the previous theorem, we obtain the following statement.

\begin{cor}
Let $\mathbb{K}$ be a field of characteristic zero and $\mathbb{F}\subset \mathbb{K}$ be a subfield. Let further $q\colon \mathbb{F}\to \mathbb{K}$ be a quadratic function, while $\varphi\colon \mathbb{F}\to \mathbb{K}$ be a homomorphism such that 
\[
q(xy)= \varphi(x)^{2}q(y)+q(x)\varphi(y)^{2}
\qquad 
\left(x, y\in \mathbb{F}\right). 
\]
Then there exists a second-order derivation $d\in \mathscr{D}_{2}(\mathbb{F}, \mathbb{K})$ such that 
\[
q(x)= \varphi\left(4xd(x)-d(x^{2})\right) 
\qquad 
\left(x\in \mathbb{F}\right). 
\]
And vice versa, that is, if d is a second-order derivation and $\varphi$ is a homomorphism, then the function $q$ given by the above formula is a quadratic function that also satisfies the above equation.
\end{cor}

\begin{proof}
In addition to the conditions of the statement, let us first consider the case when $\varphi$ is identically zero. In this case, the above equation reduces to the identity 
\[
q(xy)=0 \qquad 
\left(x, y\in \mathbb{F}\right). 
\]
Therefore, $q$ is identically zero, from which it immediately follows that it has the desired representation. 

Assume now that $\varphi$ is not the identically zero homomorphism. Then $\varphi$ is one-to-one and $\varphi^{-1}$ is also a homomorphism. 
Define the mapping $\tilde{q}$ by 
\[
\tilde{q}(x)= \varphi^{-1}(q(x)) 
\qquad 
\left(x\in \mathbb{F}\right). 
\]
Then $\tilde{q}$ will be a quadratic mapping for which we have 
\[
\tilde{q}(xy)= x^{2}q(y)+y^{2}q(x)
\qquad 
\left(x, y\in \mathbb{F}\right). 
\]
Theorem \ref{thm_add} yields that there exists a second-order derivation $d\in \mathscr{D}(\mathbb{F}, \mathbb{K})$ such that 
\[
\tilde{q}(x)= 4xd(x)-d(x^{2}) 
\qquad 
\left(x\in \mathbb{F}\right), 
\]
from which the statement of this corollary already follows.
\end{proof}

Finally, we note that the method described in the proof of Theorem \ref{thm_add} is also suitable for determining the solutions to the equation
\[
q(xy)= \varphi_{1}(x)\varphi_{2}(x)q(y)+\varphi_{1}(y)\varphi_{2}(y)q(x) 
\qquad 
\left(x, y\in \mathbb{F}\right)
\]
for the unknown quadratic function $q\colon \mathbb{F}\to \mathbb{K}$, where $\varphi_{1}, \varphi_{2}\colon \mathbb{F}\to \mathbb{K}$ are (non-identically zero) homomorphisms. 
In contrast to the previous two statements, however, in this case, the description is unfortunately not complete. 

Notice that the above equation says that the quadratic function $q$ is a moment function of degree $1$ corresponding to the exponential $q_0= \varphi_{1}\cdot \varphi_{2}$.

\begin{prop}\label{prop_moment1}
Let $\mathbb{K}$ be a field of characteristic zero and $\mathbb{F}\subset \mathbb{K}$ be a subfield. Let further $q\colon \mathbb{F}\to \mathbb{K}$ be a quadratic function, while $\varphi_{1}, \varphi_{2}\colon \mathbb{F}\to \mathbb{K}$ be  homomorphisms such that 
\[
q(xy)= \varphi_{1}(x)\varphi_{2}(x)q(y)+\varphi_{1}(y)\varphi_{2}(y)q(x) 
\qquad 
\left(x, y\in \mathbb{F}\right). 
\]
Then there exists an additive function $a\colon \mathbb{F}\to \mathbb{K}$ such that 
\[
q(x)= 2(\varphi_{1}(x)+\varphi_{2}(x))a(x)-a(x^{2})
\qquad 
\left(x\in \mathbb{F}\right), 
\]
where the additive function also fulfills
\begin{align*}
\tag{$\spadesuit$}
2\,&a\left(x\,y\,z\right)
\\
&-\left(\varphi_{1}(x)+\varphi_{2}(x)\right)\,a\left(y \,z\right) 
-\left(\varphi_{1}(y)+\varphi_{2}(y)\right)\,a\left(x\,z\right) 
-\left(\varphi_{1}(z)+\varphi_{2}(z)\right)\,a\left(x\,y\right)\\
&+ \left(\varphi_{1}(x)\,\varphi_{2}(y)+\varphi_{2}(x)\,\varphi_{1}(y)\right)\,a\left(z\right)
+ \left(\varphi_{1}(x)\,\varphi_{2}(z)+\varphi_{2}(x)\,\varphi_{1}(z)\right)\,a\left(y\right)\\
&+ \left(\varphi_{1}(y)\,\varphi_{2}(z)+\varphi_{2}(y)\,\varphi_{1}(z)\right)\,a\left(x\right) =0
\end{align*}
for all $x, y, z\in \mathbb{F}$. 
And also conversely, if $a$ is an additive function fulfilling $(\spadesuit)$ with the homomorphisms $\varphi_{1}, \varphi_{2}$, then the function $q$ given by the above formula is a quadratic function that also satisfies the above equation.
\end{prop}

\begin{proof}
Let us assume that the quadratic function $q\colon \mathbb{F}\to \mathbb{K}$ and the homomorphisms $\varphi_{1}, \varphi_{2}\colon \mathbb{F}\to \mathbb{K}$ satisfy  
\[
q(xy)= \varphi_{1}(x)\varphi_{2}(x)q(y)+\varphi_{1}(y)\varphi_{2}(y)q(x) 
\qquad 
\left(x, y\in \mathbb{F}\right). 
\]
If any of $\varphi_1$ and $\varphi_{2}$ were the identically zero homomorphisms, it would immediately follow that $q$ is identically zero. Thus, without loss of generality, we can assume that neither $\varphi_{1}$ nor $\varphi_{2}$ is the identically zero homomorphism. And then $\varphi_{1}(1)= \varphi_{2}(1)=1$. 

With $y=x$, we immediately obtain 
\[
q(x^{2})= 2\varphi_{1}(x)\varphi_{2}(x)q(x) 
\qquad 
\left(x\in \mathbb{F}\right). 
\]
Observe that the mapping 
\[
\mathbb{F}\ni x \longmapsto q(x^{2})- 2\varphi_{1}(x)\varphi_{2}(x)q(x) 
\]
is a generalized monomial of degree $4$. Thus there exists a symmetric and $4$-additive mapping $B_{4}\colon \mathbb{F}^{4}\to \mathbb{K}$ such that 
\[
B_{4}(x, x, x, x)= q(x^{2})- 2\varphi_{1}(x)\varphi_{2}(x)q(x) 
\qquad 
\left(x\in \mathbb{F}\right). 
\]
Indeed, this mapping is given by 
\begin{multline*}
B_{4}(x_{1}, x_{2}, x_{3}, x_{4})
\\=
\dfrac{1}{4!} \sum_{\sigma \in \mathscr{S}_{4}} \left[B(x_{\sigma(1)}x_{\sigma(2)}, x_{\sigma(3)}x_{\sigma(1)})- 2\varphi_{1}(x_{\sigma(1)})\varphi_{2}(x_{\sigma(2)})B(x_{\sigma(3)}, x_{\sigma(4)})\right]
\\
\left(x_{1}, x_{2}, x_{3}, x_{4} \in \mathbb{F}\right). 
\end{multline*}
Here $B\colon \mathbb{F}^{2}\to \mathbb{K}$ is the uniquely determined symmetric and bi-additive mapping for which we have $B(x, x)= q(x)$ for all $x\in \mathbb{F}$. 

With these notations, the equation in our statement means just that
\[
B_{4}(x, x, x, x)=0
\qquad 
\left(x\in \mathbb{F}\right). 
\]
Thus $B_{4}$ is identically zero on $\mathbb{F}^{4}$. 
Especially, 
\[
0=B_{4}(1, 1, 1, 1)= \left(2\,\varphi_{1}(1)\,\varphi_{2}(1)-1\right)\,B\left(1 , 1\right). 
\]
So $B(1, 1)=q(1)=0$. 

Identity 
\[
B_{4}(x, 1, 1, 1)=0
\qquad 
\left(x\in \mathbb{F}\right), 
\]
i.e., 
\begin{multline*}
\left(\varphi_{1}(1)\,\varphi_{2}(1)-1\right)\,B\left(x , 1\right)+\varphi_{1}(1)\,B
 \left(1 , 1\right)\,\varphi_{2}(x)+\varphi_{2}(1)\,B\left(1 , 1\right)\,\varphi_{1}(x)
 \\
 +\left(\varphi_{1}(1)\,\varphi_{2}(1)-1\right)\,B\left(1 , x\right)=0 
 \qquad 
 \left(x\in \mathbb{F}\right)
\end{multline*}
contains no information, since $\varphi_{1}(1)= \varphi_{2}(1)=1$ and $B(1, 1)=0$. 
Define the additive function $a\colon \mathbb{F}\to \mathbb{K}$ by
\[
a(x)=B(x, 1) 
\qquad 
\left(x\in \mathbb{F}\right)
\]
to deduce that the equation 
\[
B_{4}(x, x, 1, 1)=0 
\qquad 
\left(x\in \mathbb{F}\right)
\]
takes the form 
\[
q(x)= B(x, x)= 2(\varphi_{1}(x)+\varphi_{2}(x))a(x)-a(x^{2}) 
\qquad 
\left(x\in \mathbb{F}\right). 
\]
Therefore 
\begin{align*}
q(x^{2}) =& 2(\varphi_{1}(x^{2})+\varphi_{2}(x^{2}))a(x^{2})-a(x^{4}) \\
& 2(\varphi_{1}(x)^{2}+\varphi_{2}(x)^{2})a(x^{2})-a(x^{4}) 
\end{align*}
and also
\[
2\varphi_{1}(x)\varphi_{2}(x)q(x)=
2\varphi_{1}(x)\varphi_{2}(x) \left(2(\varphi_{1}(x)+\varphi_{2}(x))a(x)-a(x^{2})\right)
\]
holds for all $x\in \mathbb{F}$. 
This means that the additive function $a$ has to fulfill
\[
2(\varphi_{1}(x)^{2}+\varphi_{2}(x)^{2})a(x^{2})-a(x^{4}) =
2\varphi_{1}(x)\varphi_{2}(x) \left(2(\varphi_{1}(x)+\varphi_{2}(x))a(x)-a(x^{2})\right)
\qquad 
\left(x\in \mathbb{F}\right). 
\]
In other words, we have 
\begin{multline*}
a\left(x^4\right)+\left(-2\,\varphi_{2}(x)^2-2\,\varphi_{1}(x)\,\varphi_{2}(x)-2\,\varphi_{
 1}(x)^2\right)\,a\left(x^2\right)
 \\
 +\left(4\,\varphi_{1}(x)\,\varphi_{2}(x)^2+4\,\varphi
 _{1}(x)^2\,\varphi_{2}(x)\right)\,a\left(x\right)=0
 \qquad 
\left(x\in \mathbb{F}\right). 
\end{multline*}
Again, the left-hand side of this equation, as a function of the variable $x$ is a generalized monomial of degree $4$. With the application of the symmetrization method, we obtain that $a(1)=0$, and the above equation is equivalent to 
\begin{align*}
2\,&a\left(x\,y\,z\right)
\\
&-\left(\varphi_{1}(x)+\varphi_{2}(x)\right)\,a\left(y \,z\right) 
-\left(\varphi_{1}(y)+\varphi_{2}(y)\right)\,a\left(x\,z\right) 
-\left(\varphi_{1}(z)+\varphi_{2}(z)\right)\,a\left(x\,y\right)\\
&+ \left(\varphi_{1}(x)\,\varphi_{2}(y)+\varphi_{2}(x)\,\varphi_{1}(y)\right)\,a\left(z\right)
+ \left(\varphi_{1}(x)\,\varphi_{2}(z)+\varphi_{2}(x)\,\varphi_{1}(z)\right)\,a\left(y\right)\\
&+ \left(\varphi_{1}(y)\,\varphi_{2}(z)+\varphi_{2}(y)\,\varphi_{1}(z)\right)\,a\left(x\right) =0
\end{align*}
for all $x, y, z\in \mathbb{F}$ which proves the statement. 
\end{proof}

\begin{rem}
If $\varphi_{1}= \varphi_{2}= \mathrm{id}$ in equation $(\spadesuit)$, then we have 
\[
a(x^{3})-6xa(x^2)+3x^{2}a(x)=0
\qquad 
\left(x\in \mathbb{F}\right)
\]
and $a(1)=0$ for the additive function $a\colon \mathbb{F}\to \mathbb{K}$. In view of \cite[Corollary 2]{GseKisVin18} we deduce that $a\in \mathscr{D}_{2}(\mathbb{F}, \mathbb{K})$. 

Similarly, if $\varphi\colon \mathbb{F}\to \mathbb{K}$ is a non-identically zero homomorphism and we take $\varphi_{1}= \varphi_{2}= \varphi$, then equation $(\spadesuit)$ reduces to 
\[
a(x^{3})-6\varphi(x)a(x^2)+3\varphi(x)^{2}a(x)=0
\qquad 
\left(x\in \mathbb{F}\right). 
\]
This means that the mapping $d\colon \mathbb{F}\to \mathbb{K}$ defined by 
\[
d(x)= \varphi^{-1}(a(x))
\qquad
\left(x\in \mathbb{F}\right)
\]
satisfies 
\[
d(x^{3})-6xd(x^2)+3x^{2}d(x)=0
\qquad 
\left(x\in \mathbb{F}\right). 
\]
So $a= \varphi \circ d$ with an appropriate $d\in \mathscr{D}_{2}(\mathbb{F}, \mathbb{K})$. 

These two special cases allow us to conclude that in the general case (see identity $(\spadesuit)$), the function $a$ can be represented with the help of a second-order derivation and homomorphisms. Therefore, we formulate the following open problem. After answering this question, Proposition \ref{prop_moment1} will take on a more compact form.
\end{rem}

\begin{opp}
Let $\mathbb{K}$ be a field of characteristic zero, $\mathbb{F}\subset \mathbb{K}$ be a subfield and $\varphi_{1}, \varphi_{2}\colon \mathbb{F}\to \mathbb{K}$ be non-identically zero homomorphisms. Determine all those additive functions $a\colon \mathbb{F}\to \mathbb{K}$ such that $a(1)=0$ and 
\begin{align*}
2\,&a\left(x\,y\,z\right)
\\
&-\left(\varphi_{1}(x)+\varphi_{2}(x)\right)\,a\left(y \,z\right) 
-\left(\varphi_{1}(y)+\varphi_{2}(y)\right)\,a\left(x\,z\right) 
-\left(\varphi_{1}(z)+\varphi_{2}(z)\right)\,a\left(x\,y\right)\\
&+ \left(\varphi_{1}(x)\,\varphi_{2}(y)+\varphi_{2}(x)\,\varphi_{1}(y)\right)\,a\left(z\right)
+ \left(\varphi_{1}(x)\,\varphi_{2}(z)+\varphi_{2}(x)\,\varphi_{1}(z)\right)\,a\left(y\right)\\
&+ \left(\varphi_{1}(y)\,\varphi_{2}(z)+\varphi_{2}(y)\,\varphi_{1}(z)\right)\,a\left(x\right) =0
\end{align*}
for all $x, y, z\in \mathbb{F}$. 
\end{opp}

\subsection*{Quadratic functions as solutions of polynomial equations}

\begin{thm}
Let $\mathbb{K}$ be a field of characteristic zero and $\mathbb{F}\subset \mathbb{K}$ be a subfield. Let $r$ be a positive integer and suppose that for all multi-index $\alpha\in \mathbb{N}^{r}$ the mapping $q_{\alpha}\colon \mathbb{F}\to \mathbb{K}$ is quadratic such that 
\[
q_{0}(x)= x^2 
\qquad 
\left(x\in \mathbb{F}\right). 
\]
Assume further that we have 
\[
q_{\alpha}(xy)= \sum_{\beta \leq \alpha}\binom{\alpha}{\beta}q_{\beta}(x)q_{\alpha-\beta}(y) 
\qquad 
\left(x, y\in \mathbb{F}^{\times}\right)
\]
for all multi-index $\alpha\in \mathbb{N}^{r}$. 
Then there exists a sequence of $\mathbb{K}$-valued `additive functions' $a= (a_{\alpha})_{\alpha\in \mathbb{N}^{r}}$ on $\mathbb{F}^{\times}$ such that 
\[
q_{\alpha}(x)= B_{\alpha}(a(x))x^2
\]
holds for all $x\in \mathbb{F}$ and for each multi-index $\alpha\in \mathbb{N}^{r}$. Here the fact that $a_{\alpha}\colon \mathbb{F}\to \mathbb{K}$ is an `additive function' means that there exists $d_{\alpha}\in \mathscr{D}_{2}(\mathbb{F}, \mathbb{K})$ such that 
\[
a_{\alpha}(x)= 4 \frac{d_{\alpha}(x)}{x}-\frac{d_{\alpha}(x^{2})}{x^{2}} 
\qquad 
\left(x\in \mathbb{F}^{\times}\right). 
\]
\end{thm}

\begin{proof}
Let $r$ be a positive integer and suppose that for all multi-index $\alpha\in \mathbb{N}^{r}$ the mapping $q_{\alpha}\colon \mathbb{F}\to \mathbb{K}$ is quadratic such that 
\[
q_{0}(x)= x^2 
\qquad 
\left(x\in \mathbb{F}\right)
\]
and we have  also  
\[
q_{\alpha}(xy)= \sum_{\beta \leq \alpha}\binom{\alpha}{\beta}q_{\beta}(x)q_{\alpha-\beta}(y) 
\qquad 
\left(x, y\in \mathbb{F}^{\times}\right)
\]
for all multi-index $\alpha\in \mathbb{N}^{r}$. 

In view of Definition \ref{dfn_moment}, this means that the sequence of functions $(q_{\alpha})_{\alpha\in \mathbb{N}^{r}}$ forms a generalized moment sequence of rank $r$ on the group $\mathbb{F}^{\times}$, corresponding to the `exponential' $q_{0}= x^{2}\; (x\in \mathbb{F})$. Equivalently, this means that the sequence of functions $\left(q_{\alpha}/q_{0}\right)_{\alpha\in \mathbb{N}^{r}}$ forms a moment sequence of rank $r$.  Thus, we obtain from Theorem \ref{thm_moment}, that there exists  a sequence of $\mathbb{K}$-valued additive functions $a= (a_{\alpha})_{\alpha\in \mathbb{N}^{r}}$ 
 such that  for every multi-index $\alpha$ in $\mathbb{N}^{r}$ and $x$ in $\mathbb{F}^{\times}$ we have    \[
  \frac{q_{\alpha}(x)}{q_{0}(x)}=B_{\alpha}(a(x)).  
 \]
 Note that for all multi-index $\alpha\in \mathbb{N}^{r}$ `additivity' on the group $\mathbb{F}^{\times}$ means that we have 
 \[
 a_{\alpha}(xy)= a_{\alpha}(x)+a_{\alpha}(y) 
 \qquad 
 \left(x\in F^{\times}\right). 
 \]
 Thus, by Theorem \ref{thm_add} we have 
 \[
 a_{\alpha}(x)= 4 \frac{d_{\alpha}(x)}{x}-\frac{d_{\alpha}(x^{2})}{x^{2}}  
 \qquad 
 \left(x\in \mathbb{F}^{\times}\right)
 \]
 with an appropriate $d_{\alpha}\in \mathscr{D}_{2}(\mathbb{F}, \mathbb{K})$. 
\end{proof}

\begin{ackn}
The research of E.~Gselmann has been supported by project no.~K134191 that has been
implemented with the support provided by the National Research, Development and Innovation Fund of Hungary, financed under the K{\_}20 funding scheme.
\end{ackn}


\vspace{2cm}
\noindent
\textbf{Eszter Gselmann} \\
Department of Analysis\\
University of Debrecen\\
P.O. Box 400\\
H-4002 Debrecen\\
Hungary\\
e-mail: gselmann@science.unideb.hu\\
ORCID: 0000-0002-1708-2570
\vspace{1cm}

\noindent
\textbf{Mehak Iqbal}\\
Doctoral School of Mathematical and Computational Sciences\\
University of Debrecen\\
P.O. Box 400\\
H-4002 Debrecen\\
Hungary\\
e-mail: iqbal.mehak@science.unideb.hu\\
ORCID: 0000-0002-9442-3199

\end{document}